\documentclass[12pt]{article}
\usepackage{amsthm,amssymb,amsfonts,amsmath,latexsym,epsf,tikz,url}

\newtheorem{theorem}{Theorem}[section]

\newtheorem{observation}[theorem]{Observation}

\newtheorem{lemma}[theorem]{Lemma}

\begin{document}

	\title{Trees with distinguishing index equal distinguishing number plus one} 
	
	\author{
		Saeid Alikhani$^{a,}$\footnote{Corresponding author}
		\and Sandi Klav\v zar$^{b,c,d}$
		\and Florian Lehner$^{e}$
		\and Samaneh Soltani$^a$
	}
	
	\date{\today}
	
	\maketitle
	
	\begin{center}
		$^a$Department of Mathematics, Yazd University, 89195-741, Yazd, Iran\\
		{\tt alikhani@yazd.ac.ir, s.soltani1979@gmail.com}
		
		\medskip
		$^b$Faculty of Mathematics and Physics, University of Ljubljana, Slovenia\\
		{\tt sandi.klavzar@fmf.uni-lj.si}\\
		\medskip
		
		$^c$Faculty of Natural Sciences and Mathematics, University of Maribor, Slovenia\\
		\medskip
		
		$^d$Institute of Mathematics, Physics and Mechanics, Ljubljana, Slovenia\\ 
		
		\medskip
		$^e$Mathematics Institute, University of Warwick, Coventry, United Kingdom\\
		{\tt mail@florian-lehner.net}\\
		\medskip

	\end{center}
	
	\begin{abstract}
		The distinguishing number (index) $D(G)$ ($D'(G)$) of a graph $G$ is the least integer $d$ such that $G$ has an vertex (edge) labeling with $d$ labels that is preserved only by the trivial automorphism. It is known that for every graph $G$ we have $D'(G) \leq D(G) + 1$. In this note we characterize trees for which this inequality is sharp. We also show that if $G$ is a connected unicyclic graph, then $D'(G) = D(G)$. 
	\end{abstract}
	
	
	\noindent
	{\bf AMS Subj.\ Class.\ (2010)}: 05C15, 05E18

	\section{Introduction}
	\label{sec:intro}
	
	Let $G = (V(G), E(G))$ be a graph and let ${\rm Aut}(G)$ be its automorphism group. A labeling $\phi: V(G) \rightarrow [r]$ is {\em distinguishing} if no non-trivial element of ${\rm Aut}(G)$ preserves all the labels; such a labeling $\phi$ is a {\em distinguishing $r$-labeling}. More formally, $\phi$ is a distinguishing labeling if for every $\alpha\in {\rm Aut}(G)$, $\alpha\ne {\rm id}$, there exists $x\in V(G)$ such that $\phi(x) \neq \phi(\alpha(x))$. The {\em distinguishing number} $D(G)$ of a graph $G$ is the smallest $r$ such that $G$ admits a distinguishing $r$-labeling. 
	
	The introduction of the distinguishing number in 1996 by Albertson and Collins~\cite{albertson-1996} was a great success; by now about one hundred papers were written motivated by this seminal paper! The core of the research has been done on the invariant $D$ itself, either on finite \cite{Chan, immel-2017, Kim}  or infinite graphs~\cite{estaji-2017, lehner-2016, smith-2014}; see also the references therein. Extensions to group theory (cf.~\cite{klavzar-2006, wong-2010}) and arbitrary relational structures~\cite{laflamme-2010} were also investigated, as well as variations of the concept such as the distinguishing chromatic number~\cite{cavers-2013, collins-2006}. Moreover, very recently the game distinguishing number was introduced in~\cite{gravier-2017}. It is hence a bit surprising that the following variation of the distinguishing number---its edge version---was introduced only in 2015 by Kalinowski and Pil\'sniak~\cite{kalinowski-2015}. The {\em distinguishing index} $D'(G)$ of a graph $G$ is the smallest integer $d$ such that $G$ has an edge labeling with $d$ labels that is preserved only by the trivial automorphism. 
	
	Generally $D'(G)$ can be arbitrary smaller than $D(G)$, for instance if $p\geq 6$, then $D'(K_p)=2$ and $D(K_p)=p$. Conversely, there is an upper bound on $D'(G)$ in terms of $D(G)$.
	In~\cite[Theorem 11]{kalinowski-2015} (see also~\cite[Theorem 8]{lehner-2017+} for an alternative proof) it is proved that if $G$ is a connected graph of order at least $3$, then
	\begin{equation}
	\label{eq:bound}
	D'(G) \leq D(G) + 1\,.
	\end{equation}
	In this paper we give a characterisation of the trees which achieve equality. We further show that if $G$ is a connected unicyclic graph, then $D'(G) = D(G)$, showing that the inequality is never sharp for unicyclic graphs.


	
	\section{Some notation}
	\label{sec:notation}
	
	A tree $T$ is {\em unicentric} if its center (that is, the subgraph induced by the vertices of minimum eccentricity) consists of a single vertex and is {\em bicentric} otherwise. In the latter case the center is isomorphic to $K_2$ and will also be identified with its edge. 
	
	If $T$ is a bicentric tree with central edge $e = vw$, we denote by $T_v$ and $T_w$ the components of $T-e$, where $v\in T_v$ and $w\in T_w$.
	
	We will treat $T_v$ and $T_w$ as rooted trees with roots $v$ and $w$ respectively. Hence we make the following (obvious) definitions for rooted trees. An {\em automorphism of a rooted tree} is an automorphism of the underlying unrooted tree which fixes the root. Analogously, an isomorphism of rooted trees is an isomorphism which maps the root of one tree to the root of the other. An edge or vertex labeling of a rooted tree is called distinguishing, if the only automorphism (of the rooted tree) which preserves it is the identity.
	
	For rooted trees there is also a natural correspondence between vertex and edge labelings. Let $T$ be a rooted tree with root $v$. Let $f: V(T)\rightarrow [k]$ be a vertex labeling of $T$, where we use the notation $[k] := \{1,\dots,k\}$. Define $f_v':E(T) \rightarrow [k]$ as follows: If $e=xy\in E(T)$, where $d_T(y,v) < d_T(x,v)$, then set $f'_v(e) = f(x)$. Since each non-root vertex of $T$ has a unique predecessor in $T$, the labeling $f'_v$ is well-defined. It is also not hard to see, that this procedure is reversible (up to the colour of the root). We will call $f'_v$ the {\em (edge) co-labeling} of $f$ with respect to the root $v$. The following observation will be useful.
	
	\begin{observation}
		\label{obs:co-label}
		A labeling $f$ of a rooted tree is distinguishing if and only if the co-labeling $f_v'$ is distinguishing. In particular (since we can reverse the construction and the colour of the root plays no role) the distinguishing index and the distinguishing number of rooted trees are always equal.
	\end{observation}
	
	\section{Extremal trees}
	\label{sec:trees}
	
	A characterisation of trees $T$ for which $D'(G) = D(G) + 1$ was suggested in~\cite[Theorem 9]{kalinowski-2015}. While equality holds for every tree in the proposed class ${\cal B}(h,d)$, there are further trees for which the inequality is sharp as the following example demonstrates. Let $T$ be the tree which consists of a central edge with four paths of length $2$ attached to each endpoint. The tree $T$ together with a distinguishing 2-labeling demonstrating that $D(T) = 2$ is shown in Fig.~\ref{counter}. On the other hand, it is easy to verify that $D'(T) = 3$. But $T$ does not belong to the set ${\cal B}(h,d)$ which was claimed to contain all trees with $D' = D + 1$.  
	
	\begin{figure}[ht]
		\centering
		\begin{minipage}{6.5cm}
			\includegraphics[width=\textwidth]{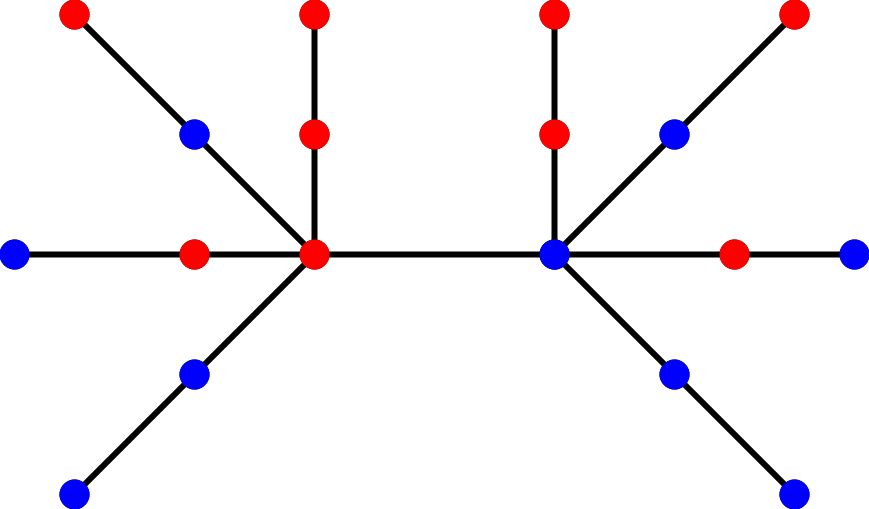}
		\end{minipage}
		\hspace{1em}
		\begin{minipage}{6.5cm}
			\includegraphics[width=\textwidth]{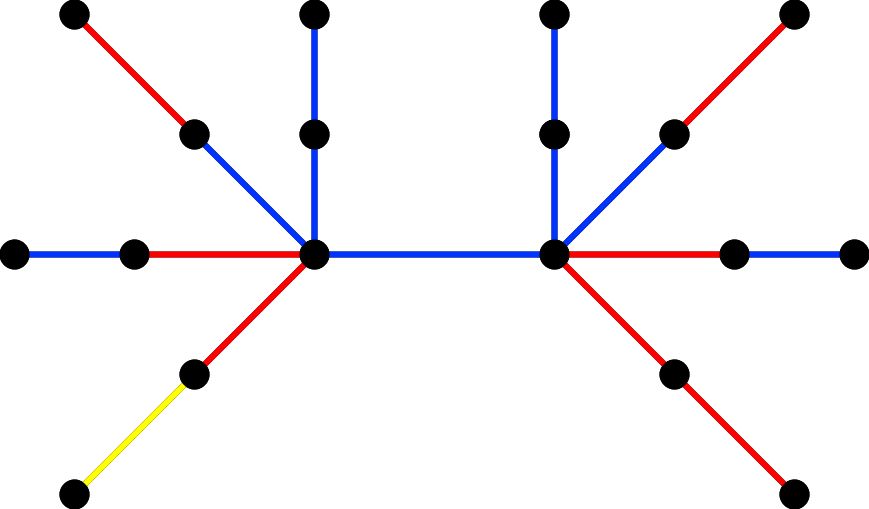}
		\end{minipage}
		\caption{A tree $T$ with $D(T) = 2$, $D'(T) = 3$ and $T\not\in{\cal B}(h,d)$}
		\label{counter}
	\end{figure}
	
	%
	%
	
	
	In this section, we give a complete characterisation of trees with $D' = D + 1$, thus correcting the flaw in ~\cite[Theorem 9]{kalinowski-2015}. Define a family ${\cal T}$ as follows. It consists of those trees $T$ of order at least $3$, for which the following conditions are fulfilled. 
	\begin{enumerate}
		\item $T$ is a bicentric tree with the central edge $e=vw$. 
		\item There is an isomorphism between the rooted trees $T_v$ and $T_w$. 
		\item There is a unique distinguishing edge-labeling of the rooted tree $T_v$ using $D(T)$ labels.
	\end{enumerate}
	The following theorem now states that the family ${\cal T}$ contains all trees with $D'(T) = D(T)+1$.
	
	\begin{theorem}
		\label{thm:trees}
		Let $T$ be a tree of order at least $3$. Then 
		$$D'(T) = \begin{cases}
		D(T) + 1 & {\rm if }\;T\in {\cal T}, \\
		D(T) & {\rm otherwise}.
		\end{cases}$$
	\end{theorem}
	
	Before we prove this theorem, we state and prove a couple of auxiliary results. Both of them are essentially contained in the proof of \cite[Theorem 9]{kalinowski-2015}.
	
	\begin{lemma}
		\label{lem:central-vertex-d-to-d'}
		If $T$ is a unicentric tree, then $D'(T) = D(T)$.  
	\end{lemma}
	
	\begin{proof}
		This follows from Observation \ref{obs:co-label} by noting that every automorphism of a unicentric tree must fix the central vertex (and thus can be seen as an automorphism of a rooted tree).
	\end{proof}
	
	\begin{lemma}
		\label{lem:central-edge-d-to-d'+1}
		If $T$ is a bicentric tree $T$, then $D(T) \le D'(T) \le D(T) +1$.  
	\end{lemma}
	
	\begin{proof}
		Throughout the proof let $e = vw$ be the central edge of $T$. Note that every automorphism of $T$ either fixes both $v$ and $w$, or swaps them.
		
		For the first inequality let $f'$ be a distinguishing edge labeling and pick a vertex labeling  $f$ such that $f'$ is the co-labeling of $f$ on $T_v$ and $T_w$. Observation \ref{obs:co-label} makes sure that no automorphism which fixes both $v$ and $w$ preserves the labeling $f$. If there is an automorphism which swaps $v$ and $w$ and preserves the labeling $f$, then this automorphism would also preserve the labeling $f'$.
		
		For the second inequality start with a distinguishing vertex labeling $f$ and label the edges of $T_v$ and $T_v$ by the corresponding co-labelings respectively. Label the central edge $e$ arbitrarily. This ensures by Observation \ref{obs:co-label} that no automorphism which fixes both $v$ and $w$ preserves the resulting edge-labeling. To ensure that the same is true for every automorphism which swaps $v$ and $w$, relabel one of the edges in $T_v$ using an additional label $D(T)+1$.
	\end{proof}
	
	As we already said in the introduction, the assertion of Lemma~\ref{lem:central-edge-d-to-d'+1} actually holds for all connected graphs. 
	
	\begin{proof}[Proof of Theorem \ref{thm:trees}]
		Observe first that the result is clearly true if $T$ is an asymmetric tree, because in this case $D(T) = D'(T) = 1$ and $T\notin {\cal T}$. Hence in the rest of the proof let $D(T)\ge 2$. 
		
		If $T$ is a unicentric tree, then $D'(T) = D(T)$ holds by Lemma~\ref{lem:central-vertex-d-to-d'}. Hence assume in the rest of the proof that $T$ is a bicentric tree with the central edge $e=vw$.
		
		By Lemma~\ref{lem:central-edge-d-to-d'+1} we get  $D(T) \le D'(T)\le D(T) + 1$.  As usual, let $T_v$ and $T_w$ be the components of $T-\{vw\}$ with $v\in T_v$ and $w\in T_w$.
		
		Suppose first that $T_v$ and $T_w$ are not isomorphic. Then for any $\alpha\in {\rm Aut}(T)$  we have $\alpha(v) = v$ and $\alpha(w)=w$. Let $f:V(T)\rightarrow [D(T)]$ be a distinguishing vertex labeling. Let $f':E(T) \rightarrow [D(T)]$ be defined as follows. Set $f'(vw) = 1$, and on $T_v$ and $T_w$ let $f'$ coincide with the co-labelings of $f|T_v$ and $f|T_w$, respectively. Then $f'$ is a distinguishing edge labeling and hence $D'(T)\le D(T)$ and equality holds by Lemma \ref{lem:central-edge-d-to-d'+1}.
		
		Suppose next that $T_v$ and $T_w$ are isomorphic and that $T_v$ admits two non-isomorphic distinguishing edge labelings (as a rooted tree) with $D(T)$ labels, say $g'$ and $g''$. Let now $f':E(T) \rightarrow [D(T)]$ be defined as follows. Set $f'(vw) = 1$, and on $T_v$ and $T_w$ let $f'$ coincide with $g'$ and $g''$, respectively. Then $f'$ is a distinguishing edge labeling and hence again $D'(T) = D(T)$. 
		
		Until now we have proved that unless $T\in {\cal T}$, then $D'(T) = D(T)$. To complete the proof we need to show that if $T\in {\cal T}$, then $D'(T) = D(T)+1$. Suppose on the contrary that this is not the case. So let $T\in {\cal T}$ be such that $D'(T) = D(T)$ and let $f': E(T) \rightarrow [D(T)]$ be a distinguishing edge labeling of $T$. Thus the restrictions of $f'$ to $T_v$ and $T_w$, respectively, are distinguishing edge labelings of $D(T)$ labels. Since $T_v$ and consequently its isomorphic copy $T_w$ admit unique $v$-distinguishing edge $D(T)$-labelings, there exists $\alpha\in {\rm Aut}(T)$ that exchanges $T_v$ with $T_w$ and preserves $f$, a contradiction. We conclude that $D'(T) = D(T)+1$.  
	\end{proof}
	\section{Unicyclic graphs}
	\label{sec:unicyclic}
	
	In this section we prove that among the unicyclic graphs the upper bound~\eqref{eq:bound} is never attained. 
	
	\begin{theorem}\label{thm:unicyclic}
		If  $G$ is a connected unicyclic graph, then $D'(G)=D(G)$. 
	\end{theorem}
	\proof
	Let $C = v_1v_2\cdots v_tv_1$ be the cycle of $G$, where $3\leq t \leq n$. Let $T_i$, $i\in [t]$, be the maximal subgraph of $G$ that contains $v_i$ and no other vertex of $C$. Then $T_i$ is a tree, consider it as a rooted tree with the root $v_i$. It is possible that $T_i$ is single vertex graph. If $G=C_t$, then the result holds because $D' = D$ holds for all cycles, see~\cite[Proposition~5]{kalinowski-2015}. The results also clearly holds if ${\rm Aut}(G)$ is trivial, hence assume in the rest of the proof that $D(G)\ge 2$ and $D'(G)\ge 2$. 
	
	We first show that $D(G)\leq D'(G)$. For this purpose let $f'$ be a distinguishing edge labeling of $G$ and define a vertex labeling $f$ as follows. On $V(T_i)\setminus\{v_i\}$, let $f$ be such that $f'$ is the co-labeling of $f$ restricted to $T_i$. Then, by Observation \ref{obs:co-label}, $f$ is a distinguishing labeling of $V(T_i)\setminus\{v_i\}$ provided that $v_i$ is fixed. If $t\geq 6$, then let $f|C$ be a distinguishing 2-labeling of $C$. If $3\leq t \leq 5$ and $f'|C$ uses at least three labels, then let $f|C$ be a distinguishing 3-labeling of $C$. In the last case we have $3\leq t \leq 5$ and $f'|C$ uses two labels. Then label the vertices of $C$ with two colors as shown in Fig.~\ref{fig:edge-to-vertex} for all possible edge labelings of $C$ with two colors. In all the cases one can verify that if an automorphism $\alpha$ of $C$ preserves $f|C$, then $\alpha$ also preserves $f'|C$. Since $f'$ is distinguishing we conclude that $f$ is also distinguishing and consequently $D(G)\le D'(G)$.  
	
	\begin{figure}[ht!]
		\begin{center}
			\includegraphics[width=0.8\textwidth]{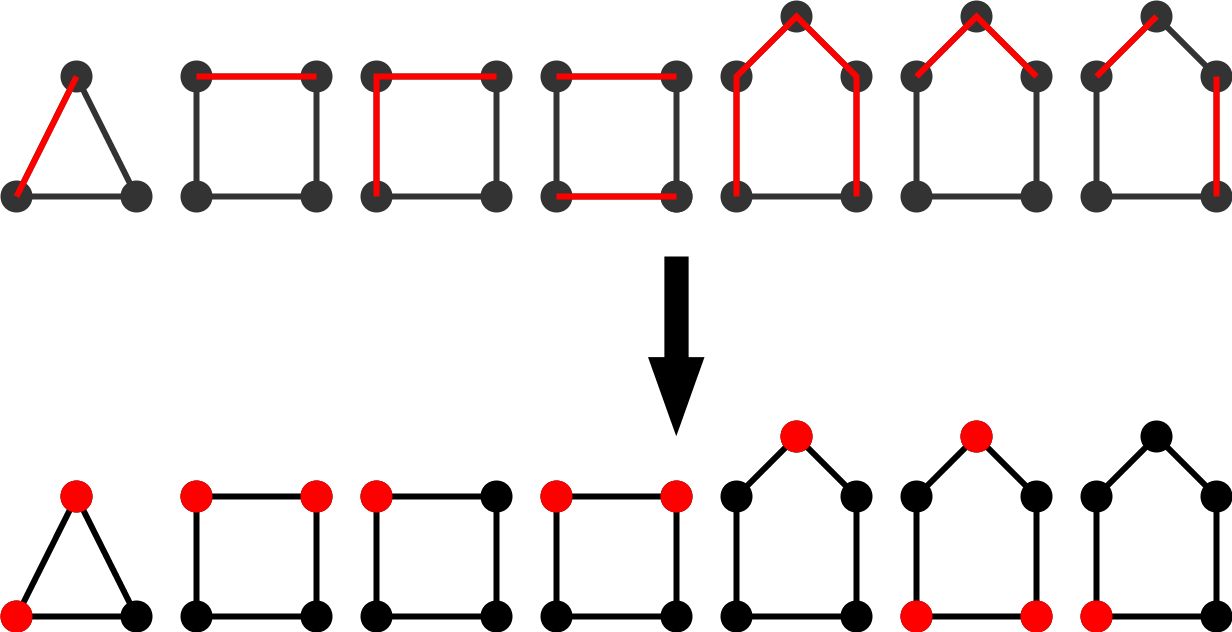}
			\caption{All non-equivalent edge $2$-labelings of $C$ and their respective transformations to vertex $2$-labelings of $C$}
			\label{fig:edge-to-vertex} 	
		\end{center}
	\end{figure}
	
	To show that also $D'(G)\leq D(G)$ holds, we proceed similarly as above. Let $f$ be a  distinguishing vertex labeling of $G$ and define an edge labeling $f'$ as follows. On each $T_i$ let $f'$ be the co-labeling of $f|T_i$ (with respect to the root $v_i$). If $t\geq 6$, then set $f'|C$ to be be a distinguishing edge 2-labeling of $C$. If $3\leq t \leq 5$ and $f|C$ uses at least three labels, then let $f'|C$ be a distinguishing edge 3-labeling of $C$. Finally, if $3\leq t \leq 5$ and $f|C$ uses two labels, then let $f'|C$ be as shown in Fig.~\ref{fig:vertex-to-edge} for all possible vertex labelings of $C$, respectively. 
	
	\begin{figure}[ht!]
		\begin{center}
			\includegraphics[width=0.8\textwidth]{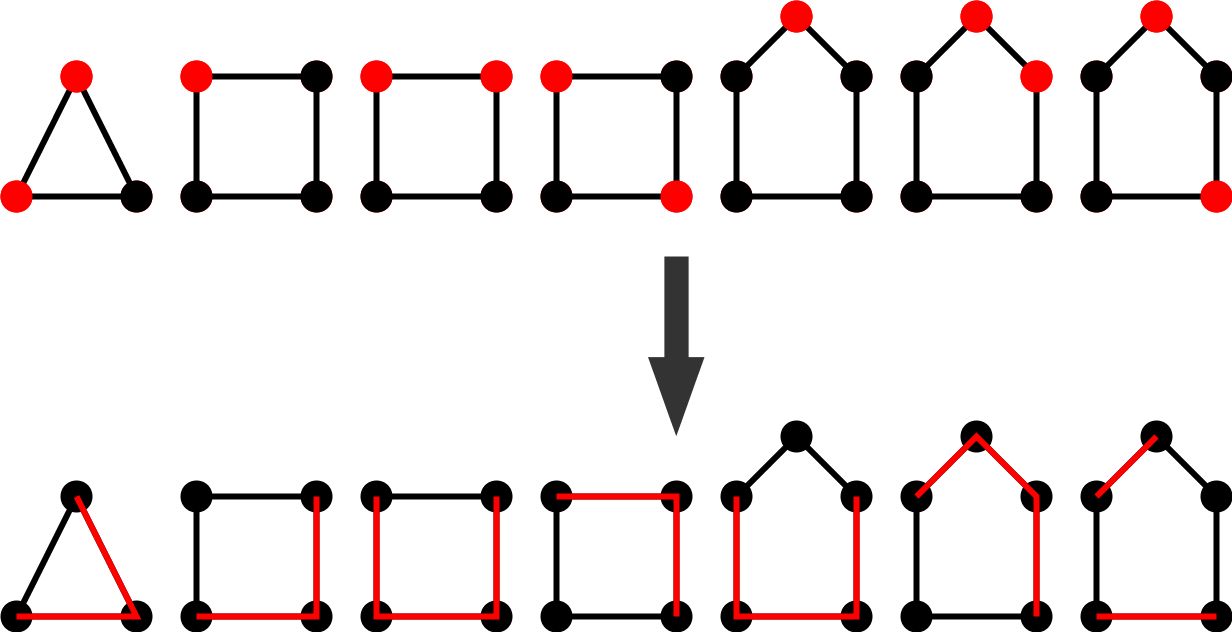}
			\caption{All non-equivalent vertex $2$-labelings of $C$ and their respective transformations to edge $2$-labelings of $C$}
			\label{fig:vertex-to-edge}  	
		\end{center}
	\end{figure}
	
	Again we can verify that if an automorphism $\alpha$ of $C$ preserves $f'|C$, then $\alpha$ also preserves $f|C$. So $f'$ is distinguishing and we conclude that $D'(G)\le D(G)$.
	\qed    
	
	\section*{Acknowledgements}
	
	Sand Klav\v{zar} acknowledges the financial support from the Slovenian Research Agency (research core funding No.\ P1-0297). Florian Lehner acknowledges the support of the Austrian Science Fund (FWF), grant number J 3850-N32

\end{document}